\documentclass[reqno,a4paper,11pt]{amsproc}
\usepackage{amsmath,amscd,amsfonts,amssymb}
\usepackage{calrsfs}
\usepackage{color}
\usepackage{nicefrac}
\usepackage{mathtools}

\usepackage{bbm}

\definecolor{myblue}{rgb}{0.09,0.32,0.44} %22-84-113
\usepackage[unicode=true,pdfusetitle,
 bookmarks=true,bookmarksnumbered=false,bookmarksopen=false,
 breaklinks=true,pdfborder={0 0 0},pdfborderstyle={},backref=false,colorlinks=true,linkcolor=myblue,citecolor=myblue,urlcolor=blue]
 {hyperref}

\usepackage{breakurl}

\newcommand\N{\mathbb{N}}
\newcommand\R{\mathbb{R}}
\newcommand\Z{\mathbb{Z}}

\newcommand\eps{\varepsilon}

\renewcommand{\P}{\mathbb{P}}
\newcommand{\E}{\mathbb{E}}

\newcommand\var{\operatorname{Var}}

\renewcommand\le{\leqslant}
\renewcommand\ge{\geqslant}

\theoremstyle{plain}
\newtheorem{theorem-main}{Theorem}
\newtheorem{corollary-main}{Corollary}
\newtheorem{theorem}{Theorem}
\newtheorem{proposition}[theorem]{Proposition}
\newtheorem{lemma}[theorem]{Lemma}

\newtheorem*{theorem-a}{Theorem A}
\newtheorem*{theorem-b}{Theorem B}
\newtheorem*{convention}{Convention}

\theoremstyle{definition}
\newtheorem*{definition*}{Definition}
\newtheorem{definition}{Definition}
\newtheorem*{remark*}{Remark}

\def\EC{\mathcal{E}}
\def\CP{\mathcal{P}}
\def\go{{G\setminus\{o\}}}
\def\ho{{H\setminus\{o\}}}
\def\gcal{{\mathcal{G}}}

\ifpdf
\def\afs#1#2{\href{#1}{\nolinkurl{#2}}}
\else
\def\afs#1#2{\burlalt{#1}{#2}}
\fi

\title[Power-law decay of weights in the 2D VRJP]{Power-law decay of weights and recurrence of the two-dimensional VRJP}

\author{Gady Kozma}
\address{Gady Kozma\hfill\break
	Department of Mathematics and Computer Science\\
	The Weizmann Institute of Science\\
	Rehovot, 76100, Israel.}
\email{gady.kozma@weizmann.ac.il}

\author{Ron Peled}
\address{Ron Peled\hfill\break
	School of Mathematical Sciences\\
	Tel Aviv University\\
	Tel Aviv, 69978, Israel.}
\email{peledron@tauex.tau.ac.il}
\urladdr{http://www.math.tau.ac.il/~peledron}

\begin{document}
\begin{abstract}
  The vertex-reinforced jump process (VRJP) is a form of self-interacting random walk in which the walker is biased towards returning to previously visited vertices with the bias depending linearly on the local time at these vertices. We prove that, for any initial bias, the weights sampled from the magic formula on a two-dimensional graph decay at least at a power-law rate. Via arguments of Sabot and Zeng, the result implies that the VRJP is recurrent in two dimensions for any initial bias.
\end{abstract}

\maketitle

\section{Introduction}
In this paper we study an interacting stochastic process known as
the \emph{vertex-reinforced jump process} (VRJP for short) in two
dimensions using a technique known as the \emph{Mermin--Wagner theorem}.
We start by describing VRJP, our object of study.

\subsection{The vertex-reinforced jump process}
VRJP was first studied in \cite{DV02} as a continuous-time version
of \emph{linearly edge reinforced random walk} (LRRW), a process studied
earlier by Diaconis and Coppersmith (unpublished, 1987) who noted
that it has an interesting property not shared by other reinforced
random walks: partial exchangeability. Partial exchangeability for
a discrete-time process means that the probability of any particular
path depends only on the number of times each edge was crossed, and
not on the order in which this happened. This property allows, via
a soft argument \cite{DF80}, to conclude that LRRW is in fact a random walk in
random environment (RWRE) and using a more elaborate argument
to get a formula for the distribution of the environment, known fondly
as ``the magic formula''. See \cite{MOR08} for the history of the
magic formula. For VRJP the picture is slightly different. The process
is not a (continuous time) RWRE
as stated. Instead, the process becomes a RWRE after a \emph{time change} and then has its own magic formula. A hint of the magic
formula for VRJP appeared in \cite{DV04} but the full picture was
only revealed by Sabot and Tarr\`es \cite{ST15}. The magic formula
will be stated exactly below, in \S\ref{subsec:Exact}.

A second special property of VRJP is the connection to \emph{supersymmetry}.
We will not attempt to describe supersymmetry in details in this short
introduction, but roughly it postulates a symmetry between \emph{fermions}
and \emph{bosons}. The specific supersymmetric model relevant to VRJP
is the \emph{hyperbolic sigma model}, defined by Zirnbauer~\cite{Zir91, DFZ92} (see also~\cite{DSZ10, DS10}). The hyperbolic sigma model has two fermions and two
bosons at each vertex, with an interaction that enjoys a hyperbolic symmetry. Integrating both fermionic fields and one of the
bosonic fields leads to a single field, let us denote it by $u$,
but with a complicated interaction term. It was discovered in \cite{ST15} that
$e^{u}$ has exactly the same distribution as the environment described
by the magic formula for the VRJP, establishing a link between these
two topics. Supersymmetry brings a new set of tools to the problem,
but the most relevant to us is the \emph{Ward identity.} It states
that $\mathbb{E}e^{u_{x}}=1$ always.

In this paper we study the VRJP in two dimensions. We show that the
environment of its RWRE representation decays at least like a power law, namely $\E e^{\frac{1}{2}u_x}\le|x|^{-c}$
where the constant $c$ may depend on the initial weight $a$ (see
exact definitions below). For LRRW this was proved by Merkl and Rolles
\cite{MR08}. This result is not sharp for small $a$. In this case
it was known \cite{ST15, ACK14} that in fact $u$ decays exponentially.
The true decay rate for large $a$ thus remains open. We do not know
if it is really a power law (and hence a transition of \emph{Kosterlitz--Thouless type} occurs) or rather if the decay is exponential for all
$a$. This is related to the question of \emph{asymptotic freedom}
in quantum field theory, but this introduction is too short to cover
these connections.

While this paper was written (which, unfortunately, took much too much time), Sabot gave an alternative proof of this result, see \cite{S19} (see also~\cite{DMR14} for the quasi one-dimensional case).

\subsection{Exact definitions and statements}\label{subsec:Exact}

\begin{definition}\label{def:VRJP}
Let $G$ be a finite graph, let $o$ be a vertex of $G$ and let $W:E(G)\to[0,\infty)$ be a function.
The \emph{vertex-reinforced jump process} (VRJP) on $G$ with initial vertex $o$ and weights $W$ is a continuous-time process $(Y_{t})_{t\ge 0}$ on the vertices of $G$ defined as follows: $Y_{0}:=o$
and at every $t>0$, $Y$ jumps from its current position $x$ to a neighbour
$y$ with rate $W_{xy}L_{y}(t)$ where $L$ is the local time:
\begin{equation}\label{eq:local time def}
L_{y}(t):=1+\int_{0}^{t}\mathbbm{1}\{Y_{s}=y\}\,ds.
\end{equation}
The \emph{time-changed VRJP} on $G$ with initial vertex $o$ and weights $W$ is the process $(Z_s)_{s\ge 0}$ obtained by setting $Z_s = Y_{D^{-1}(s)}$ where $D:[0,\infty)\to[0,\infty)$ is the (random) increasing bijection
\begin{equation}\label{eq:half}
  D(t) := \sum_{x\in G} (L_x^2(t) - 1).
\end{equation}
\end{definition}
We remark at this point that the time change is not as bad as it looks on first sight: it applies at each vertex essentially independently, a fact that we use below in the proof of Lemma \ref{lem:VRJP}.

\begin{convention}The $u$ in the theorem below, and more generally any function defined on $\go$ is considered to be zero on $o$.
\end{convention}

\begin{theorem}[``The magic formula'']\label{thm:magic}
Let $G$ be a finite connected graph, let $o$ be a vertex of~$G$ and let $W:E(G)\to(0,\infty)$.
\begin{enumerate}
  \item The function $\rho:\mathbb{R}^{\go}\to\mathbb{R}$ below is a probability density function:
    \begin{equation}
    \begin{split}
    \rho(u)&:=\frac{1}{(2\pi)^{(|G|-1)/2}}\exp\Big(-\sum_{x\in G}u_x\Big)\\
    &\qquad\exp\Big(-\sum_{\smash{\{x,y\}\in E(G)}}W_{xy}(\cosh(u_x-u_y)-1)\Big)
      \sqrt{D(W,u)}\label{eq:density rho}
    \end{split}
    \end{equation}
    where $D(W,u)$ is any diagonal minor of the matrix $A=(a_{xy}:x,y\in G)$
    given by
    \begin{equation}\label{eq:A matrix}
    a_{xy}:=\begin{cases}
    -W_{xy}e^{u_x+u_y} & x\ne y, \{x,y\}\in E(G)\\
    0&x\ne y, \{x,y\}\notin E(G)\\
    -\sum_{z\neq x}a_{xz} & x=y.
    \end{cases}
    \end{equation}
    (Note that the convention $u_o=0$ was used in the sums in~\eqref{eq:density rho}.)
  \item Sample $u$ randomly from the density $\rho$. Let $(Z_{s})_{s\ge 0}$ be a continuous-time random walk on $G$, with $Z_0 := o$, which transitions from $x$ to a neighbour $y$ with rate $\frac{1}{2}W_{xy}e^{u_y-u_x}$. Then $(Z_{s})$ (considered after averaging over the randomness in $u$) is distributed as the time-changed VRJP on $G$ with initial vertex $o$ and weights $W$.
\end{enumerate}
\end{theorem}
See Sabot and Tarr\`es \cite[Theorem 2]{ST15}. The result there
is stated for $2\sinh^{2}(\frac{1}{2}(u_x-u_y)$ instead of for $\cosh(u_x-u_y)-1$
but this is of course the same. It is stated under the condition
$\sum u_{i}=0$ whereas our normalisation is $u_o=0$ but, again,
this is the same: the measure on $\{\sum u_{i}=0\}$ used in \cite{ST15}
is not the volume measure but simply the measure one gets by fixing
$u_{x}=0$ for an arbitrary vertex $x$ (see the comment immediately
after \cite[Theorem 2]{ST15}); and the $\rho$ is unchanged, except
the term $e^{u_o}$ in \cite{ST15} which becomes our $\exp(-\sum u_x)$.

The result of this paper is that in a two-dimensional graph the weights decay at least at a power-law rate. The result is local, only the structure at the vicinity of the point of interest is used. Here is the exact formulation.

Here and below, for vertices $x,y$ in a graph write $d(x,y)$ for their graph distance and, for integer $L\ge 0$, denote the closed ball of radius $L$ around $x$ by $B(x,L):=\{y\colon d(x,y)\le L\}$. We denote by $c$ and $C$ positive absolute constants whose value might change from line to line or even within the same line. We use $c$ for constants which are ``small enough'' and $C$ for constants which are ``large enough''. We use $c(\dotsb)$ and $C(\dotsb)$ for constants that depend on some parameters.
\begin{theorem}\label{thm:main}
  There exists $C,c(a)>0$ such the following holds for each $a>0$: Let $L\in\N$. Let $G$ be a finite connected graph with a distinguished vertex $o$ and assume that $B(o,L)$ is isomorphic to the ball $B((0,0),L)$ in $\Z^2$. Let $W:E(G)\to(0,\infty)$ satisfy that $W|_{B(o,L)}\equiv a$. Let $u$ be sampled from the density \eqref{eq:density rho} with respect to $G$ and $W$. Then for every $x\in B(o,2L)$,
  \begin{equation*}
    \P(u_x \ge -c(a)\log(d(o,x)))\le \frac{C}{d(o,x)^{c(a)}},
  \end{equation*}
  and
  \begin{equation}\label{eq:ehalfux}
    \E\left(e^{\frac{1}{2}u_x}\right) \le \frac{C}{d(o,x)^{c(a)}}.
  \end{equation}
Further, for every $a_0>0$ there exists $c(a_0)>0$ such that $c(a)\ge c(a_0)/a$ for $a>a_0$.
\end{theorem}

Part of the motivation for proving a local result comes from its application in Sabot--Zeng \cite{SZ15}. They proved recurrence of two-dimensional VRJP, conditioned on Theorem \ref{thm:main} in the specific case of wired boundary conditions. Here is the exact formulation. For an integer $L\ge 1$ let $G_L$ be the graph with vertex set $\{-L,\dotsc,L\}^2\cup\{\delta_L\}$, with vertices in $\{-L,\dotsc,L\}^2$ connected by an edge if they differ by exactly one in exactly one coordinate and with $\delta_L$ adjacent to every vertex $(x_1, x_2)$ for which either $|x_1|=L$ or $|x_2|=L$ (or both). Correspondingly, for a real $a>0$, $W_L:E(G_L)\to \R$ satisfies $W_L(e)=a$ for all edges except the edges connecting $\delta_L$ with $(x_1,x_2)$ having both $|x_1|=|x_2|=L$ for which $W_L(e)=2a$ (as these edges result from two edges of $\Z^2$ when identifying the vertices of $\Z^2$ adjacent to $V_L$ into the single vertex $\delta_L$). Let $o = (0,0)\in G_L$. Then Remark 7 in \cite{SZ15} says that if \eqref{eq:ehalfux} holds for this graph then VRJP on the whole of $\Z^2$ is recurrent. Since this falls under our Theorem \ref{thm:main} this gives a proof of recurrence of VRJP. We remark that a proof of a weaker notion of recurrence was given recently in \cite{BHS18}.

\subsection{Overview of the proof}
The core of the proof is an argument of Mermin--Wagner type, so let us start with a short discussion of this approach. For physicists, the Mermin--Wagner theorem states that continuous symmetries cannot be spontaneously broken in a system with short-range interactions in dimension $2$ or lower (see e.g.\ \cite[p. 198]{S06}).
Every use of the Mermin--Wagner approach starts with a perturbation
argument: a calculation (usually easy to do) shows that it is possible
to take one instance of the field $u$ and then deform it so that
the local deformation is small, small enough to have low energetic cost, while the overall deformation is significant.
For example, take the field $u$ to have density $\exp\big(-\sum (\nabla u)^2\big)$, i.e.\ a two-dimensional lattice Gaussian free field. The continuous group of symmetries in this case is simply the symmetries taking $u_x\mapsto u_x+C$ for some constant $C$, which preserves the density. Consider $u$ in the discrete box $\{-L,-L+1,\ldots,L\}^2$, normalized so that $u_{(0,0)}=0$. The perturbation argument entails comparing $u_{x}$
to $u_{x}^{\pm}:=u_{x}\pm\tau$, with $\tau$ chosen, e.g., as $\log(|x|+1)/\sqrt{\log L}$. The energy of $u$ is necessarily close to either that of $u^+$ or that of $u^-$ (since $(\nabla u)^2 - \frac{1}{2}((\nabla u^+)^2 + (\nabla u^-)^2) = (\nabla\tau)^2$ and the sum of the last term is uniformly bounded in $L$ by the choice of $\tau$) but overall the fields diverge by $\sqrt{\log L}$, which is
significant. One concludes that the fluctuations of $u$ must grow without bound as $L$ increases (specifically, $\var u_x\ge c\log L$ at vertices at distance $L$ from the origin).
There are multiple approaches to harness
the perturbation argument; the reader may find a discussion with references in \cite[page 4]{MP15}.

Let us now describe how to apply the above approach to the VRJP model. Recall that the density of $u$~\eqref{eq:density rho} (``the magic formula'') is proportional to
\begin{equation*}
  \exp\left(-\sum_{\smash{\{x,y\}\in E(G)}}W_{xy}U(u_x - u_y) + F(u)\right)
\end{equation*}
with
\begin{equation*}
\begin{split}
  U(s)&:=\cosh(s)-1,\\
  F(u)&:=-\sum_{x\in G}u_x + \frac{1}{2}\log(D(W,u)).
\end{split}
\end{equation*}
The idea is to use an argument of Mermin--Wagner type to lower bound the fluctuations of $u$ by (a constant multiple of) the fluctuations of the Gaussian free field, normalized to be zero at $o$, whose density is proportional to
\begin{equation*}
  \exp\left(-\sum_{\smash{\{x,y\}\in E(G)}}W_{xy}(u_x - u_y)^2\right).
\end{equation*}
On a two-dimensional graph, with $W_{xy}\equiv a$, this yields that
\begin{equation}\label{eq:variance lower bound}
  \var(u_x) \ge \frac{c}{a}\log(d(o,x))
\end{equation}
and corresponding Gaussian lower bounds on the tail behavior.
As a separate input, we use that the field $u$ is known not to be too big. Specifically, the Ward identity of Theorem~\ref{thm:Ward identity} shows that
\begin{equation}\label{eq:u not too big}
  \E(\exp(u_x))=1.
\end{equation}
In order to avoid a contradiction between~\eqref{eq:variance lower bound} (and the corresponding tail bounds) and~\eqref{eq:u not too big}, the value of $u_x$ must typically be small. For a quantitative inequality, recall that if $Z$ is a Gaussian random variable with mean $m$ and variance $\sigma^2$ then
\begin{equation*}
  \E(e^Z) = e^{m + \frac{1}{2}\sigma^2}.
\end{equation*}
Thus if $u_x$ were Gaussian then~\eqref{eq:variance lower bound} and~\eqref{eq:u not too big} would imply that
\begin{equation*}
  \E(u_x) \le -\frac{c'}{a}\log(d(o,x))
\end{equation*}
Theorem~\ref{thm:main} states quantitative results of this flavor.

In a different context, the idea that fluctuation lower bounds plus an a priori input that the field is not large could be used to prove that the field is typically small was used by Schenker in~\cite{S09} following a suggestion of Aizenman. A version of the Mermin--Wagner method was also applied by Merkl and Rolles~\cite{MR08} in proving power-law decay of the weights for the LRRW.

There are two main obstacles to the application of a Mermin--Wagner type argument to the VRJP model. First, the field $u$ does not have short-range interactions due to the presence of the determinant term. This is handled by noting that the function $F(u)$ above is \emph{log-convex} (Lemma~\ref{lem:log convexity}) and can thus be discarded in comparing the energy of $u$ with the energy of its perturbations $u^\pm$ (as $\sqrt{F(u+\tau)F(u-\tau)}\ge F(u)$ for any $\tau$). Second, the Mermin--Wagner method is easiest to implement for gradient fields whose interaction function $U$ is twice-continuously differentiable with $\sup U''<\infty$ (see, e.g., \cite[\S{} 1.1]{MP15} or~\cite[\S{} 2.6]{PS17}). As the hyperbolic cosine function does not satisfy this bound we need to resort to a more sophisticated version of the argument, based on ideas of Richthammer~\cite{R07} and developed in~\cite{MP15} (see the `addition algorithm' of \S~\ref{sec:the addition algorithm}). A price to pay is that an additional a priori input is required: We need to show that for some sufficiently large constant $K$, the random set of edges on which the gradient of $u$ exceeds $K$ in absolute value is sparse in an appropriate probabilistic sense. For other models, such an input was established either using reflection positivity~\cite{MP15} or using symmetries of the state space~\cite{CAP17}. Here, this input is proved by the approach of \cite{ACK14}, namely, by considering together the VRJP and RWRE pictures for the field $u$ (see \S~\ref{sec:percolation}). The obtained constant $K$ is uniform in the weight $a$ for $a$ bounded away from zero.

\section{Inputs on the weight distribution}
Let $G$ be a finite connected graph, $W:E(G)\to(0,\infty)$ and $o\in G$. We describe two inputs on the density~\eqref{eq:density rho} of the weight vector $u$.

The first input is known as a \emph{Ward identity}.
\begin{theorem}\label{thm:Ward identity}
If $u$ is sampled from the density~\eqref{eq:density rho} then $\mathbb{E}e^{u_x}=1$ for all $x\in G$.
\end{theorem}
See \cite{DMR17}, formula (5.26) (formula (5.19) in the arXiv version). The
model is defined slightly differently in \cite{DMR17}, there is an
extra function $s:\go\to\R$ and an extra term in the
density, $\exp(-\tfrac{1}{2}s^{t}As)$. Our model is the marginal distribution of $u$ in the model of \cite{DMR17}: integrating
the term $\exp(-\tfrac{1}{2}s^{t}As)$ gives a term $(2\pi)^{-(|G|-1)/2}D(W,u)^{-1/2}$
which, together with the term $D(W,u)$ in \cite{DMR17}, gives our
term $(2\pi)^{-(|G|-1)/2}\sqrt{D(W,u)}$.

The second input is a simple log-convexity property which will be key to the application of Mermin--Wagner type techniques in the proof of Theorem~\ref{thm:main}. This is well-known (see, e.g., Disertori--Spencer--Zirnbauer~\cite[Remark 2.3]{DSZ10} where the proof is attributed to David Brydges), but for completeness we provide a proof.

\begin{lemma}\label{lem:log convexity}
Let $A$ be the matrix given by~\eqref{eq:A matrix}. Let
$D(W,u)$ be the determinant of any diagonal minor of $A$. Then $\sqrt{D(W,u)}$ is a log-convex function of $u$.
\end{lemma}

\begin{proof}
Since $A$ is a Laplacian matrix, i.e.\ a symmteric matrix with nonpositive
off-diagonal entries and rows summing to 0, we may apply the matrix-tree
theorem \cite[theorem 1.19]{HMM08}. It gives
\[
D := D(W,u)=\sum_{T\in\mathcal{T}}\prod_{\{x,y\}\in E(T)}W_{xy}e^{u_{x}+u_{y}}
\]
where $\mathcal{T}$ is the set of all spanning trees $G$ (recall that a spanning tree of a graph is a subgraph which contains all vertices and some of the edges, and is connected and cycle-free).
Rearranging gives
\[
D=\sum_{f:G\to\mathbb{Z}}a_{f}\prod_{x\in G}e^{f(x)u_{x}}
\]
for some coefficients $a_{f}\ge0$ (all but finitely many of which are zero). Any such sum is log-convex: indeed, for each $x,y\in G$,
\begin{align*}
  D^{2}\frac{\partial^{2}}{\partial u_{x}\partial u_{y}}(\log D)
  & =\frac{\partial^{2}D}{\partial u_{x}\partial u_{y}}D
    -\frac{\partial D}{\partial u_{x}}\frac{\partial D}{\partial u_{y}}\\
  & =\sum_{f,g}a_{f}a_{g}(f(x)f(y)-f(x)g(y))\prod_{z}e^{(f(z)+g(z))u_{z}}
\end{align*}
and the symmetry between $f$ and $g$ allows to write the sum as
\begin{multline*}
  \frac{1}{2}\sum_{f,g}a_{f}a_{g}(f(x)f(y)-f(x)g(y)+g(x)g(y)-g(x)f(y))
    \prod_{z}e^{(f(z)+g(z))u_{z}}=\\
  \frac{1}{2}\sum_{f,g}a_{f}a_{g}(f(x)-g(x))(f(y)-g(y))\prod_{z}e^{(f(z)+g(z))u_{z}}.
\end{multline*}
To see that the resulting matrix is positive semi-definite, let $\mu$
be some test vector and write
\[
D^{2}\sum_{x,y}\mu_{x}\mu_{y}\frac{\partial^{2}}{\partial u_{x}\partial u_{y}}(\log D)
=\frac{1}{2}\sum_{f,g}a_{f}a_{g}\Big(\sum_{x}\mu_{x}(f(x)-g(x))\Big)^{2}
  \prod_{z}e^{(f(z)+g(z))u_{z}}.
\]
Since this is nonnegative, the lemma is proved.
\end{proof}

\section{Comparing to percolation}\label{sec:percolation}

Let $G$ be a finite connected graph, $W:E(G)\to(0,\infty)$ and $o\in G$. Let $a>0$ and let $H\subset G$ be some induced subgraph such that $W|_{E(H)}\equiv a$.
Let $u$ be sampled from the density~\eqref{eq:density rho}. In this section we show that the (random) set of edges $\{x,y\}$ where $|u_x - u_y|$ is large is sparse in a suitable sense.

A random set of edges is called an \emph{$\eps$-percolation} if each edge of the underlying graph is present in the random set with probability $\eps$, independently between different edges. Our proof entails consideration of two $\eps$-percolations, which may be dependent among themselves.

\begin{proposition}\label{prop:perco} For every $\eps>0$ there exists $K=K(\eps,a)$ such that
  the set $\{\{x,y\}\in E(H)\colon |u_x-u_y|\ge K\}$ is dominated by a union of two (dependent) $\eps$-percolations. Further, $\sup\{K(\eps,a)\colon a>a_0\}<\infty$ for each $a_0>0$.
\end{proposition}
Write $\vec E(H)$ for the set of directed edges of $H$. The proof revolves around an ``estimator'' for $\exp(u_x-u_y)$, $(x,y)\in \vec E(H)$, which we denote by $Q_{xy}$ (note that $Q$ is not symmetric). We define $Q$ via the $Z$ process; recall that it has two equivalent definitions, via the VRJP picture (see Definition~\ref{def:VRJP}) and via the RWRE picture (see Theorem \ref{thm:magic}).

Define $Q_{xy}$ to be the local time spent by $Z$ at $x$ up to the first jump from $x$ to $y$.

\begin{lemma} \label{lem:RWRE}For $\eps>0$ let $K_1=a/2\eps$. Then the set $\{(x,y)\in \vec E(H):\exp(u_x-u_y)\ge K_1 Q_{xy}\}$ is dominated by $\eps$-percolation.
\end{lemma}
\begin{lemma}\label{lem:VRJP}
  For every $\eps>0$ there exists $K_2=K_2(\eps,a)$ such that the set $\{(x,y)\in \vec E(H):Q_{xy}\ge K_2\}$ is dominated by $\eps$-percolation. Further, one may assume that $K_2\le C(\eps)(a^{-1}+a^{-2})$.
\end{lemma}
Proposition \ref{prop:perco} following immediately from these two lemmas, with $K_{\textrm{Proposition \ref{prop:perco}}}=\log(K_{1}K_{2})$ and with $\eps_{\textrm{Lemmas \ref{lem:RWRE} and \ref{lem:VRJP}}}=\frac12 \eps_{\textrm{Proposition \ref{prop:perco}}}$ (the $\frac12$ is needed because both lemmas give directed $\eps$-percolation, while the proposition is about undirected percolation, and projecting a directed $\eps$-percolation to an undirected one gives a $(2\eps-\eps^2)$-percolation). Let us therefore move to the proofs of these lemmas.
\begin{proof}[Proof of Lemma \ref{lem:RWRE}]
Examine the RWRE picture. We claim that conditioned on the environment $u$, $Q_{xy}\cdot(\frac12 W_{xy} \exp(u_y-u_x))$ is an i.i.d.\ field of exponential random variables of rate 1. This implies that the unconditioned field (after integrating over $u$) satisfies the same. The lemma follows, as an exponential random variable $T$ with rate $1$ satisfies $\P(T\le \eps)=\int_0^\eps e^{-x}dx\le \eps$.

To see the above claim we recall a method for implementing a continuous-time random walk. For each directed edge $(x,y)$, denote the jump rate from $x$ to $y$ by $\rho_{xy}$ and associate to $(x,y)$ an independent Poisson process with intensity $\rho_{xy}$. The walk is then defined by the rule that a jump from $x$ to $y$ occurs at times $t$ for which (i) the walker is at the vertex $x$ just before time $t$, and (ii) an event of the Poisson process of $(x,y)$ occurs at time $L_x(t)$, where $L_x(t)$ is the local time accumulated at the vertex $x$ by time $t$.
It is a standard fact that the walk defined in this way indeed has the correct distribution. With this representation, it becomes clear that if $Q_{xy}$ is the local time spent by the walk at $x$ up to its first jump to $y$ then the $(Q_{xy})$ are independent and each $Q_{xy}$ has an exponential distribution with rate $\rho_{xy}$.
%Define the walk by the rule that when the walker is at a vertex $x$ and an event of the Poisson process occurs for a directed edge $(x,y)$ then the walker jumps from $x$ to $y$.
\end{proof}
\begin{proof}[Proof of Lemma \ref{lem:VRJP}]
Examine the VRJP picture, and denote by $q_{xy}$ the time spent by $Y$ in $x$ up to the first jump from $x$ to $y$. The instantaneous jump rate from $x$ to $y$, $W_{xy}L_y(t)$, is always larger than $a$ (as $L_y\ge 1$, see~\eqref{eq:local time def}). Hence $q_{xy}$ is dominated by a field of i.i.d.\ exponential random variables with rate $a$. In particular,
\begin{equation}\label{eq:exponential of rate a}
  \P\left(q_{xy}>\frac{t}{a}\right)\le e^{-t}.
\end{equation}
We now claim that $Q_{xy}=q_{xy}^2+2q_{xy}$, which will then imply the lemma. Indeed, let $t_i$ be the $i^{\textrm{th}}$ time that $Y$ enters $x$, let $t'_i$ be the $i^{\textrm{th}}$ time $Y$ exits $x$, and let $k$ be the minimal $i$ for which $Y$ exits $x$ towards $y$ at time $t_i'$, so $q_{xy}=\sum_{i=1}^k (t_i'-t_i)$. The RWRE picture makes it clear that all the $t_i$ and $t_i'$, as well as $k$, are almost surely finite. Recall the time change function $D$ from \eqref{eq:half}, so $Q_{xy}=\sum_{i=1}^kD(t_i')-D(t_i)$. But between $t_i$ and $t_i'$ the sum defining $D$ changes only at $x$ so
  \[
  D(t_i')-D(t_i)=L_x^2(t_i')-L_x^2(t_i)=L_x^2(t_i')-L_x^2(t_{i-1}')
  \]
where we define $t_0'=0$. Hence
  \begin{align*}
  Q_{xy}&=\sum_{i=1}^kD(t_i')-D(t_i)=\sum_{i=1}^kL_x^2(t_i')-L_x^2(t_{i-1}')^2\\
  &=L_x^2(t_k')-L_x^2(0)=(1+q_{xy})^2-1=q_{xy}^2+2q_{xy}
  \end{align*}
as needed. The lemma follows.
\end{proof}

\section{The addition algorithm}\label{sec:the addition algorithm}

The final ingredient used in our proof is the following, so called \emph{addition algorithm}, which is introduced in~\cite{MP15} following earlier work of Richthammer~\cite{R07}.

The input to the addition algorithm is a finite, connected graph $H$, a function $\tau:H\to[0,\infty)$ and a constant $K$. Its output is two bijections $T^+, T^-$ on $\R^H$ such that $T^\pm(\varphi)$ is an approximation of $\varphi\pm \tau$, chosen in a way that preserves the gradients of $\varphi$ whenever the latter are larger than $K$. The exact formulation is below.

While the explicit description of the addition algorithm is not long or difficult (see~\cite[\S{} 2.2]{MP15}), we refrain from giving it here and instead list the properties of the algorithm which we require. The list follows the properties in~\cite[\S{} 2.1]{MP15}, with the exception of property \ref{property:addition_lower_bound} for which we provide the stronger statement given in~\cite[Proposition 2.7]{MP15}, and with a few differences in formulation which are explained following the list.

Let $H$ be a finite connected graph with a distinguished vertex $o$. We sometimes write $v\sim w$ to denote that $\{v,w\}\in E(H)$. Let $\tau:H\to[0,\infty)$, $\tau_o=0$ and
$K>0$ be given. The addition algorithm defines a pair of measurable mappings $T^+, T^-:\R^\ho\to\R^\ho$ related by the
equality
\begin{equation}\label{eq:T_+_T_-_relation}
  T^+(\varphi) - \varphi = \varphi -
  T^-(\varphi),\quad\varphi\in\R^{\ho},
\end{equation}
and satisfying the following properties:
\renewcommand{\theenumi}{{(\roman{enumi})}}
\renewcommand{\labelenumi}{(\roman{enumi})}
\begin{enumerate}
  \item (bijections) \label{property:one-to-one_onto}$T^+$ and $T^-$ are one-to-one and onto.
  \item (add at most $\tau$) \label{property:maximal_increment}For every $\varphi\in\R^{\ho}$ and every $v\in H$,
  \begin{equation}\label{eq:T^+_T^-_increment_range}
    0\le T^+(\varphi)_v - \varphi_v = \varphi_v - T^-(\varphi)_v\le \tau_v.
  \end{equation}
  \item (gradient preservation) \label{property:Lipschitz_preservation} For every $\varphi\in\R^{\ho}$ and every $(v,w)\in E(H)$,
  \begin{align*}
    &|\varphi_v - \varphi_w|\ge 2K\implies T^\pm(\varphi)_v -
    T^\pm(\varphi)_w = \varphi_v -
    \varphi_w,\\
    &|\varphi_v - \varphi_w|< 2K\implies |T^\pm(\varphi)_v -
    T^\pm(\varphi)_w| < 2K.
  \end{align*}
\end{enumerate}
The properties stated so far do not exclude the possibility that
$T^+$ is the identity mapping (implying the same for $T^-$ by
\eqref{eq:T_+_T_-_relation}). The next property shows that
$T^+(\varphi) - \varphi$ is close to $\tau$ under certain
restrictions on the set of edges on which $\varphi$ changes by at least $K$. We require a few definitions.

Recall that $d$ stands for graph distance, here on the graph $H$.
The next two definitions concern the Lipschitz properties of $\tau$.
\begin{align}
\tau'(v,k)&:=\max\{\tau_v - \tau_w\colon w\in H,\, d(v,w)\le
k\}\label{eq:tau_prime_def},\\
L(\tau,K)&:=\max\left\{k\colon d(v,w)< k\implies |\tau_v-\tau_w|\le \tfrac12 K\right\}\label{eq:L_tau_eps_def}
\end{align}
(the $'$ in $\tau'$ is supposed to remind the reader of differentiation).
In the following definitions we consider the connectivity properties
of the subset of edges on which $\varphi$ changes by more than
$K$. For $\varphi\in\R^H$ define
\begin{equation}\label{eq:EC_def}
  \EC(\varphi):=\{(v,w)\in E(H)\colon |\varphi_v-\varphi_w|\ge K\}
\end{equation}
and write, for a pair of vertices $v,w\in H$,
\begin{equation}\label{eq:EC_connectivity_def}
v\xleftrightarrow{\EC(\varphi)} w\;\text{ if $v$ is connected to $w$
by edges of $\EC(\varphi)$},
\end{equation}
where we mean in particular $v\xleftrightarrow{\EC(\varphi)} v$ for all $v\in H$. Let
\begin{align}\label{def:rAndm}
  r(\varphi, v)&:= \max\{d(v,w)\colon w\in H,\,
  v\xleftrightarrow{\EC(\varphi)}w\},\\
  M(\varphi)&:=\max\{d(v,w)\colon v,w\in H,\,v\xleftrightarrow{\EC(\varphi)}w\}.\label{eq:M phi def}
\end{align}
\begin{enumerate}
\setcounter{enumi}{3}
  \item (add close to $\tau$) \label{property:addition_lower_bound}
  For any $\varphi\in\R^\ho$ satisfying $M(\varphi)\le
  L(\tau,K) - 2$,
  \begin{equation*}
    \tau_v - \tau'(v,r(\varphi,v))\le T^+(\varphi)_v - \varphi_v\le \tau_v\quad\text{
    for all $v\in H$.}
  \end{equation*}
\end{enumerate}
Our final property regards the change of measure induced by the
mappings $T^+$ and $T^-$. We bound the Jacobians of these mappings
when the subgraph $\EC(\varphi)$ does not contain many large
connected components.

\begin{enumerate}
\setcounter{enumi}{4}
  \item (Jacobians) \label{property:Jacobian_estimate} There exist measurable functions $J^+,J^-:\R^\ho\to[0,\infty)$
      satisfying
  \begin{equation}\label{eq:Jacobian_formula_properties_section}
\int g(T^+(\varphi))J^+(\varphi)\, d\varphi = \int
g(T^-(\varphi))J^-(\varphi)\, d\varphi = \int g(\varphi)
\,d\varphi
\end{equation}
for every measurable $g:\R^{\ho}\to[0,\infty)$ (where $d\varphi$ stands for Lebesgue measure on $R^{\ho}$). These functions satisfy the estimate
  \begin{equation}\label{eq:Jacobian estimate}
    \sqrt{J^+(\varphi)J^-(\varphi)}\ge \exp\left(-\frac{1}{K^2}\sum_{v\in H} \tau'\left(v,1+\max_{w\sim
v}r(\varphi,w)\right)^2\right)
  \end{equation}
  at every $\varphi\in\R^\ho$ for which $M(\varphi)\le L(\tau,K)-2$.
\end{enumerate}
\renewcommand{\theenumi}{{\arabic{enumi}}}
\renewcommand{\labelenumi}{(\arabic{enumi})}

For easier comparison with \cite{MP15} let us explain the few differences between the way the result is formulated here and there.
\begin{enumerate}
\item In \cite{MP15} there is an additional parameter $\eps$. We set this $\eps$ to $\frac12$.
\item The parameter $K$ does not appear in \cite{MP15}. There, the constant $2K$ appearing in property~\ref{property:Lipschitz_preservation} is replaced by $1$.
The version here is achieved by dividing $\varphi$ and $\tau$ by $2K$, applying the addition algorithm of~\cite{MP15} and then multiplying back by $2K$.
  \item Our $L$ is defined slightly differently than in \cite{MP15}, with $L=L_{\textrm{\cite{MP15}}}+2$.
  \item In~\cite{MP15} there is no distinguished vertex $o$ on which the functions $\tau$ and $\varphi$ are assumed to be zero. In addition, the Jacobians are shown to satisfy a stronger property than~\eqref{eq:Jacobian_formula_properties_section}, allowing to fix the functions $\varphi$ to arbitrary values on vertices where $\tau$ is zero. Here, for simplicity, we restricted to the case that $\tau$ and $\varphi$ are fixed to zero at $o$ as this is the only case we will use.
\end{enumerate}

\section{Proof of the main result}\label{sec:proof of main theorem}
In this section we combine the previous ingredients to prove Theorem~\ref{thm:main}.

Let $L\in\N$. Let $G$ be a finite connected graph with a distinguished vertex $o$ and assume that $B(o,L)$ is isomorphic to the ball $B((0,0),L)$ in $\Z^2$. Let $W:E(G)\to(0,\infty)$ satisfy that $W|_{B(o,L)}\equiv a$. Let $u$ be sampled from the density \eqref{eq:density rho} with respect to $G$ and $W$. We need to show that there exist $C,c(a)>0$ so that for any $a>0$ and any $x\in B(o,2L)$,
\begin{align}
    &\P(u_x \ge -c(a)\log(d(o,x)))\le \frac{C}{d(o,x)^{c(a)}},\label{eq:u_x greater than u_0 prob}\\
    &\E\left(e^{\frac{1}{2}u_x}\right) \le \frac{C}{d(o,x)^{c(a)}}\label{eq:fractional exponential moment expect}
\end{align}
and that $c(a)$ can be taken to be at least $c(a_0)/a$ for all $a>a_0$. We assume throughout the following that $d(o,x)$ (and thus also $L$) is at least a large absolute constant as for each fixed $d(o,x)$ we may take $C$ large enough and $c(a_0)$ small enough to make~\eqref{eq:u_x greater than u_0 prob} trivial and make~\eqref{eq:fractional exponential moment expect} follow from the Ward identity (Theorem~\ref{thm:Ward identity}).

The following is our main lemma, which shows that $u_x$ must be either larger than $c(a)\log d(o,x)$ or smaller than $-c(a)\log d(o,x)$, with high probability. The theorem follows from it, see page \pageref{pg:proof}, by a simple application of the Ward identity (which is also used in the proof of the lemma).
\begin{lemma}\label{lem:u x u o deviation bound}
Let $a_0>0$. There exist $C, c(a_0)>0$ such that for every $a>a_0$,
  \begin{equation*}
    \P\left(|u_x|\le \frac{c(a_0)}{a}\log\Big(\frac{1}{4}\sqrt{d(o,x)}\Big)\right)\le C\, d(o,x)^{-c(a_0)/a}.
  \end{equation*}
\end{lemma}
The rest of the section is devoted to proving the lemma and deducing Theorem~\ref{thm:main}. Throughout we fix $a_0>0$ and assume that $a>a_0$.

We wish to use the addition algorithm from the previous section and to this end we need to specify the graph $H$, target function $\tau$ and constant $K$. Let $H$ be the induced subgraph of $G$ on the closed ball $B(o,\frac{1}{2}d(o,x))=\{y\in G\colon d(o,y)\le\frac{1}{2}d(o,x)\}$, so that $H$ is a ball in $\Z^2$ regardless of the choice of $x$. The choice to make the radius of the ball proportional to $d(o,x)$ is made in order for the parameter $M(\varphi)$ appearing in the addition algorithm to typically not be too large.
The parameter $K$ will be fixed using Lemma~\ref{lem:bad event prob estimates} below to a value depending only on $a_0$.
To specify $\tau$ we introduce a parameter $\lambda$ which will be fixed later (following~\eqref{eq:final probability estimate}) to a value of the form $c(a_0)/a$. Define $\tau:H\to[0,\infty)$ by
\begin{equation}\label{eq:tau choice}
\tau_y:=\begin{cases}
0 & d(o,y)< \sqrt{d(o,x)}\\
\lambda\log\Big(\frac{d(o,y)}{\sqrt{d(o,x)}}\Big) & \sqrt{d(o,x)}\le d(o,y)<\frac{1}{4}d(o,x)\\
\lambda\log\Big(\frac{1}{4}\sqrt{d(o,x)}\Big) & \frac{1}{4}d(o,x)\le d(o,y)\le\frac{1}{2}d(o,x).
\end{cases}
\end{equation}
The reason for taking $\tau$ to be $0$ up to a large distance from $o$ is to increase the size of the parameter $L(\tau,K)$ defined in~\eqref{eq:L_tau_eps_def}. Indeed,
\begin{equation}\label{eq:L lower bound}
  L(\tau,K)\ge \frac{K}{2\lambda}\sqrt{d(o,x)},
\end{equation}
as nearest-neighbour differences satisfy $\max_{y\sim z}|\tau_y - \tau_z|\le \lambda d(o,x)^{-1/2}$.

For these $H, \tau$ (and the parameter $K$ to be fixed below), the addition algorithm produces mappings $T^\pm:\R^\ho\to\R^\ho$ and the associated $J^\pm:\R^\ho\to[0,\infty)$. We define extensions of these maps on the whole of $\R^\go$ as follows. First, $\bar{T}^\pm:\R^\go\to\R^\go$ are defined by
\begin{equation}\label{eq:u pm def}
u^\pm_y \coloneqq \bar{T}^\pm(u)_y \coloneqq \begin{cases}
T^\pm(u|_H)(y) & d(o,y)\le \frac{1}{2}d(o,x)\\
u_y\pm\lambda\log\Big(\frac{1}{4}\sqrt{d(o,x)}\Big) & \textrm{otherwise.}
\end{cases}
\end{equation}
Second, the maps $\bar{J}^\pm:\R^\go\to[0,\infty)$ are defined by $\bar{J}^\pm(u) = J^\pm(u|_H)$. It is simple to check that the extension of Property~\ref{property:Jacobian_estimate} of the addition algorithm holds, namely that
\begin{equation}\label{eq:19half}
  \int g(\bar T^\pm(\varphi))\bar J^\pm(\varphi)\, d\varphi
  = \int g(\varphi)\,d\varphi
\end{equation}
where $d\varphi$ now stands for Lebesgue measure on $\R^{G\setminus\{o\}}$.
It is convenient to introduce a notation for the actual increments due to the addition algorithm
\begin{equation}\label{eq:increments}
i_y\coloneqq u^+_y - u_y
\end{equation}
where the reader should keep in mind that (as will be shown) $i$ is close to $\tau$ on $H$ in a suitable sense. In particular, by~\eqref{eq:T^+_T^-_increment_range}, $i_y=0$ whenever $\tau_y = 0$, i.e.,
\begin{equation}\label{eq:i y zero}
  i_y = 0\text{ when }d(o,y)<\sqrt{d(o,x)}.
\end{equation}

We require some control over the Jacobians and increments resulting from the addition algorithm and this is provided by the following definition and lemma.
\begin{definition}\label{def:G}
  For a constant $\sigma$ we define a ``good'' event $\gcal=\gcal(\sigma)\subseteq\R^\go$ as the set of all $u$ satisfying that
  \begin{align}
  &\sqrt{J^+(u)J^-(u)}\ge \frac{1}{d(o,x)^{\sigma\lambda^2}},\label{eq:Jacobian in E}\\
  &i_y = \tau_y\quad\text{for all $y$ satisfying $d(o,y) = \left\lfloor\tfrac{1}{2}d(o,x)\right\rfloor$},\label{eq:increment on boundary}\\
  &\sum_{y\sim z} (i_y - i_z)^2 \le \sigma\lambda^2\log d(o,x).\label{eq:sum of squares of increments}
\end{align}
\end{definition}
\begin{lemma}\label{lem:bad event prob estimates}
  Suppose $\lambda\le 1$. There exist absolute constants $C,c,\sigma$ and a choice of $K$ as a function solely of $a_0$ for which $\P(\gcal(\sigma))\ge 1-C\exp(-c\linebreak[0]d(o,x)^{1/4})$.
\end{lemma}

Lemma \ref{lem:bad event prob estimates} follows in a straightforward manner from the ``two dependent percolations'' picture and properties of the addition algorithm so we postpone its proof. Continuing with the proof of lemma \ref{lem:u x u o deviation bound}, denote by $\mathcal{H}$ the event to be estimated in the lemma, i.e., $\mathcal{H}$ is the set of all $u$ satisfying $|u_x|\le \frac{\lambda}{3}\log\Big(\frac{1}{4}\sqrt{d(o,x)}\Big)$. Fix $K$ and $\sigma$ as in lemma \ref{lem:bad event prob estimates}, define $\gcal$ using this $\sigma$ and define the event
\begin{equation*}
  \mathcal{I}:=\mathcal{H}\cap\mathcal{G}.
\end{equation*}
Let $\rho$ be the density of the field $u$ (``the magic formula'') as given in~\eqref{eq:density rho}.
The proof of Lemma~\ref{lem:u x u o deviation bound} makes use of the Ward identity and the fact that $\rho$ has the form
\begin{equation}\label{eq:magic formula essentials}
  \rho(u) = \rho_1(u)\cdot\rho_2(u)\cdot\rho_3(u)
\end{equation}
with
\begin{equation*}
  \rho_1(u) = \exp\left(-a\sum_{\smash{\{y,z\}\in E(H)}} \cosh(u_y - u_z)\right)
\end{equation*}
(using that $W|_{B(o,L)}\equiv a$ by assumption),
with $\rho_2$ a function of the \emph{gradients} of $u$ on the edge set $E(G)\setminus E(H)$ and with $\rho_3$ a log-convex function. We take $\rho_3=\exp(-\sum u_x)\sqrt{D(W,u)}$, a product of a log-linear term and a term whose log-convexity is justified by Lemma~\ref{lem:log convexity}.

Our analysis starts with the quantity
\[
I:=\int_{\mathcal{I}}\sqrt{\rho(u^+)\rho(u^-)\bar{J}^+(u)\bar{J}^-(u)}\,du
\]
for which we proceed to establish upper and lower bounds (again, $u\in\R^\go$ and the integration is with respect to the Lebesgue measure on $\R^\go$).
On the one hand, by the Cauchy-Schwarz inequality and~\eqref{eq:19half},
\begin{multline}\label{eq:I upper bound}
\qquad\qquad  I\le \left(\int_{\mathcal{I}}\rho(u^+)\bar{J}^+(u)du \,
      \int_{\mathcal{I}}\rho(u^-)\bar{J}^-(u)du\right)^{1/2}\\
  \stackrel{\textrm{(\ref{eq:19half})}}{=}
  \left(\P(u\in \bar{T}^+(\mathcal{I}))\,\P(u\in \bar{T}^-(\mathcal{I}))\right)^{1/2}.\qquad\qquad
\end{multline}
(recall that $\bar{T}^+, \bar{T}^-$ mean the transformations mapping $u$ to $u^+, u^-$, i.e., the transformations defined on the whole graph $G$ rather than just on the subgraph $H$.)
On the other hand, by property \eqref{eq:Jacobian in E} of the good event $\mathcal{G}$ (which contains $\mathcal{I}$),
\begin{equation}\label{eq:I pre lower bound}
I\ge \frac{1}{d(o,x)^{\sigma\lambda^2}} \int_{\mathcal{I}}\sqrt{\rho(u^+)\rho(u^-)}du
\end{equation}
and we proceed to find a lower bound for the integrand. We study the three factors in~\eqref{eq:magic formula essentials} separately. First, by log-convexity and the relation~\eqref{eq:T_+_T_-_relation} of the addition algorithm,
\begin{equation}\label{eq:rho3 lower bound}
  \sqrt{\rho_3(u^+)\rho_3(u^-)}
  \ge \rho_3\big(\tfrac{1}{2}(u^+ + u^-)\big)
  \stackrel{\mathclap{\textrm{\eqref{eq:T_+_T_-_relation}}}}{=}
  \rho_3(u).
\end{equation}
Second, by~\eqref{eq:u pm def}, we have $i_y = u^+_y - u_y = \lambda\log\Big(\frac{1}{4}\sqrt{d(o,x)}\Big)$ for all $y\in G\setminus H$. In addition, by property~\eqref{eq:increment on boundary} of the good event $\mathcal{G}$ we have $i_y = \lambda\log\Big(\frac{1}{4}\sqrt{d(o,x)}\Big)$ for all $y\in H$ with $d(o,y) = \left\lfloor\frac{1}{2}d(o,x)\right\rfloor$, i.e., those $y\in H$ which are endpoints of an edge in $E(G)\setminus E(H)$. Thus, on the event $\mathcal{I}$ the gradients of $u^+, u^-$ on $E(G)\setminus E(H)$ equal the gradients of $u$ there and we obtain
\begin{equation}\label{eq:rho2 lower bound}
  \rho_2(u^+) = \rho_2(u^-) = \rho_2(u).
\end{equation}
Lastly, we calculate
\begin{multline*}
\sqrt{\rho_1(u^+)\rho_1(u^-)}\\
= \exp\left(-a\sum_{\smash{\{y,z\}\in E(H)}} \frac{1}{2}\Big(\cosh(u_y - u_z + (i_y - i_z)) + \cosh(u_y - u_z - (i_y - i_z))\Big)\right)
\end{multline*}
To obtain a simpler expression for the summands we note that by property \ref{property:Lipschitz_preservation} of the addition algorithm, $i_y = i_z$ when $|u_y - u_z|\ge 2K$. A second-order Taylor expansion of $\cosh$ thus gives
\[
\frac{1}{2}(\cosh(u_y - u_z + (i_y - i_z)) + \cosh(u_y - u_z - (i_y - i_z))) \le \cosh(u_y - u_z) + C(K) (i_y - i_z)^2
\]
with $C(K)>0$ solely a function of $K$. For simplicity, denote all constants that depend only on $a_0$ by $C(a_0)$, in particular the $C(K)$ above. In conclusion,
\begin{equation}\label{eq:rho1 lower bound}
\sqrt{\rho_1(u^+)\rho_1(u^-)} \ge \rho_1(u) \exp\left(-C(a_0)a\sum_{\smash{\{y,z\}\in E(H)}} (i_y-i_z)^2\right).
\end{equation}
Putting together~\eqref{eq:rho3 lower bound}, \eqref{eq:rho2 lower bound} and \eqref{eq:rho1 lower bound} we thus have on $\mathcal{I}$ that
\begin{equation*}
  \sqrt{\rho(u^+)\rho(u^-)} \ge \rho(u) \exp\left(-C(a_0)a\sum_{\smash{\{y,z\}\in E(H)}} (i_y-i_z)^2\right).
\end{equation*}
Plugging this bound back into~\eqref{eq:I pre lower bound} and using property~\eqref{eq:sum of squares of increments} of the good event $\mathcal{G}$,
\begin{equation}\label{eq:I lower bound}
I \ge \frac{1}{d(o,x)^{\sigma\lambda^2(C(a_0)a+1)}} \int_\mathcal{I} \rho(u)du \ge \frac{1}{d(o,x)^{C(a_0)\lambda^2 a}} \P(u\in \mathcal{I}),
\end{equation}
where in the second inequality we compensated for removing the $+1$ and the $\sigma$ from the power by increasing $C(a_0)$ (recall that $a>a_0$ and that $\sigma$ is an absolute constant). Combining~\eqref{eq:I lower bound} with the upper bound~\eqref{eq:I upper bound} brings us to the key inequality
\begin{equation}\label{eq:key inequality}
\begin{split}
  \P(u\in \mathcal{I})&\le d(o,x)^{C(a_0)\lambda^2 a}\left(\P(u\in \bar{T}^+(\mathcal{I}))\P(u\in \bar{T}^-(\mathcal{I}))\right)^{1/2}\\
  &\le d(o,x)^{C(a_0)\lambda^2 a}\,\P(u\in \bar{T}^+(\mathcal{I}))^{1/2}.
\end{split}
\end{equation}
We develop the right-hand side of the inequality. As $\mathcal{I}=\mathcal{H}\cap\mathcal{G}$ we have
\begin{equation}\label{eq:E and B}
  \P(u\in \bar{T}^+(\mathcal{I}))\le \P(u\in \bar{T}^+(\mathcal{H})).
\end{equation}
Further recalling that $\mathcal{H}$ is the set of $u$ satisfying $|u_x|\le \frac{\lambda}{3}\log\Big(\frac{1}{4}\sqrt{d(o,x)}\Big)$ and that $\bar{T}^+$ is given by~\eqref{eq:u pm def} we have that
\begin{equation}\label{eq:T plus B}
\begin{split}
  \P(u\in \bar{T}^+(\mathcal{H}))&= \P\left(\Big|u_x - \lambda\log\Big(\frac{1}{4}\sqrt{d(o,x)}\Big)\Big|\le \frac{\lambda}{3}\log\Big(\frac{1}{4}\sqrt{d(o,x)}\Big)\right)\\
  &\le \P\left(u_x\ge \frac{2}{3}\lambda\log\Big(\frac{1}{4}\sqrt{d(o,x)}\Big)\right).
\end{split}
\end{equation}
Putting together~\eqref{eq:E and B} and~\eqref{eq:T plus B} and making use of Markov's inequality and the Ward identity (Theorem~\ref{thm:Ward identity}) now shows that
\begin{equation*}
  \P(u\in \bar{T}^+(\mathcal{I}))
  \le \exp\Big(-\frac{2}{3}\lambda\log\Big(\frac{1}{4}\sqrt{d(o,x)}\Big)\Big)
  = \left(\frac{16}{d(o,x)}\right)^{\lambda/3}.
\end{equation*}
Combining this inequality with \eqref{eq:key inequality} we get
\begin{equation}\label{eq:final probability estimate}
\P(u\in\mathcal{I})\le d(o,x)^{C(a_0)\lambda^2 a-\lambda/6}\cdot 4^{\lambda/3}.
\end{equation}
We see that for $\lambda\le c(a_0)/a$ for some positive $c(a_0)$ sufficiently small, the power becomes negative (and we may also ensure that $\lambda\le 1$, to satisfy the assumption of Lemma~\ref{lem:bad event prob estimates}, by taking $c(a_0)\le a_0$). Fix $\lambda$ to such a value. The proof of Lemma~\ref{lem:u x u o deviation bound} is now finished since, by Lemma~\ref{lem:bad event prob estimates},
\begin{equation*}
  \P(u\in \mathcal{H}) \le \P(u\in\mathcal{I}) + \P(u\notin \mathcal{G})\le Cd(o,x)^{-c(a_0)/a}+C\exp(-cd(o,x)^{1/4})
\end{equation*}
and the second term is negligible.

\begin{proof}[Proof of Theorem \ref{thm:main}]
\phantomsection\label{pg:proof}
Fix $a_0>0$ and suppose that $a>a_0$. By Lemma~\ref{lem:u x u o deviation bound} there exist $C, c_1(a_0)>0$ so that
\begin{equation}\label{eq:probability for u_x close to u_o}
    \P\left(|u_x|\le t\right)\le C\, d(o,x)^{-c_1(a_0)/a}.
\end{equation}
with
\begin{equation*}
  t:=\frac{c_1(a_0)}{a}\log\Big(\frac{1}{4}\sqrt{d(o,x)}\Big).
\end{equation*}
The probability that $u_x$ is large can be bounded by the Ward identity (Theorem \ref{thm:Ward identity}) and Markov's inequality:
\begin{equation}\label{eq:WardMarkov}
  \P(u_x>t)\le\frac{\E(e^{u_x})}{e^t}=e^{-t}.
\end{equation}
Together \eqref{eq:probability for u_x close to u_o} and \eqref{eq:WardMarkov} show \eqref{eq:u_x greater than u_0 prob}. To further deduce~\eqref{eq:fractional exponential moment expect} we write
\begin{equation*}
  \E\left(e^{\frac{1}{2}u_x}\right)
  =\E\left(e^{\frac{1}{2}u_x}\mathbbm{1}_{u_x\in I_1}\right) + \E\left(e^{\frac{1}{2}u_x}\mathbbm{1}_{u_x\in I_2}\right) + \E\left(e^{\frac{1}{2}u_x}\mathbbm{1}_{u_x\in I_3}\right)
\end{equation*}
where $I_1 := (-\infty, -t)$, $I_2 := [-t,s]$, $I_3 := (s,\infty)$, $s := \min\{t,\frac{c_1(a_0)}{a} \log(d(o,x))\}$ and $t$ is as before. We trivially have
\begin{equation*}
  \E\left(e^{\frac{1}{2}u_x}\mathbbm{1}_{u_x\in I_1}\right)\le e^{-t/2}.
\end{equation*}
Using~\eqref{eq:probability for u_x close to u_o} we have
\begin{equation*}
  \E\left(e^{\frac{1}{2}u_x}\mathbbm{1}_{u_x\in I_2}\right)
  \le e^{s/2}\P(u_x\in I_2)
  \stackrel{\mathclap{\textrm{\eqref{eq:probability for u_x close to u_o}}}}{\le}
  C\frac{d(o,x)^{c_1(a_0)/2a}}{d(o,x)^{c_1(a_0)/a}}
  =\frac{C}{d(o,x)^{c_1(a_0)/2a}}.
\end{equation*}
For $I_3$ we use the Ward identity to get
\[
\E(e^{\frac12 u_x}\mathbbm{1}_{u_x>s})
\le e^{-s/2}\E(e^{u_x}\mathbbm{1}_{u_x>s})
\le e^{-s/2}\E(e^{u_x}) = e^{-s/2}.
\]
The inequality~\eqref{eq:fractional exponential moment expect} follows by combining the last four displayed equations and plugging the definitions of $t$ and $s$.
\end{proof}

\section{Properties of a union of percolations}\label{sec:percolation2}

In this section we discuss two specific quantitative ways in which the union of $\eps$-percolations is sparse, which are required for the proof of lemma \ref{lem:bad event prob estimates}. Our analysis takes the underlying graph to be the whole square lattice as this suffices for our purposes.

Let $\CP_1,\CP_2$ be two (dependent) $\eps$-percolations on $\Z^2$. Write $\CP$ for their union. Define the radius of connected components in $\CP$ by
\begin{equation*}
  r(y):=\max\{d(y,z)\colon \text{$z$ is connected to $y$ by edges in $\CP$}\},\quad y\in\Z^2.
\end{equation*}
\begin{lemma}\label{lem:subcriticality}
  There exists $\eps_0>0$ such that if $\eps\le\eps_0$ then
  \begin{equation*}
    \P(r(y)\ge k)\le e^{-k}\quad\text{for $y\in\Z^2$ and integer $k\ge 1$}.
  \end{equation*}
\end{lemma}
\begin{proof}
  The event in question entails the existence of a simple path $\gamma$ with $k$ edges of $\CP$ starting from $y$. In this case there is some $i\in\{1,2\}$ such that at least $\lceil k/2\rceil$ of the edges of $\gamma$ are in $\CP_i$. For a fixed $\gamma$ and $i$ this probability can be bounded by $\eps^{k/2}2^k$. Summing over $\gamma$ (for which there are less than $4^k$ possibilities) and $i$ gives
  \begin{equation*}
    \P(r(y)\ge k)\le 2\cdot 8^k\cdot \eps^{k/2}.
  \end{equation*}
For $\eps$ sufficiently small, this is smaller than $e^{-k}$ for all $k\ge 1$.
\end{proof}

\begin{lemma}\label{lem:prob estimate on sum of radiuses}
There exist $\varepsilon_0, C, c>0$ such that if $\varepsilon\le\varepsilon_{0}$ then for all $\ell>0$,
\begin{equation}\label{eq:prob estimate on sum of radiuses}
\mathbb{P}\Big(\sum_{y\colon \ell\le d(o,y)\le\ell^{2}}\frac{r(y)^{2}}{d(o,y)^{2}}\ge\log\ell\Big)\le Ce^{-c\,\ell^{1/2}}.
\end{equation}
\end{lemma}
(the value $\frac{1}{2}$ can be improved easily, but this is not
useful for us).
\begin{proof}
We assume that $\ell$ is sufficiently large as otherwise the claim is trivial. Denote the sum in~\eqref{eq:prob estimate on sum of radiuses} by $S$. We proceed to upper bound $S$ by sums involving simpler random variables. Let $M\ge 1$ be a parameter and write
\[
S_{M}:=\sum_{\ell\le d(o,y)\le\ell^{2}}\frac{r(y)^2}{d(o,y)^{2}}\mathbbm{1}\{r(y)\in[M,2M)\},
\]
so that $S=\sum_{m=0}^{\infty}S_{2^{m}}$.
Observe that
\begin{equation}\label{eq:SM by at least M}
S_{M}\le 4M^2\sum_{\ell\le d(o,y)\le\ell^{2}}\frac{1}{d(o,y)^{2}}\mathbbm{1}\{r(y)\ge M\}.
\end{equation}
Denote by $\mathcal{E}(y,i,M)$ the event that there is a simple path $\gamma$ from $y$ to some $z$ with $d(y,z)=M$ with at least half of the edges of $\gamma$ in $\CP_i$. As in the proof of Lemma~\ref{lem:subcriticality} we have
\begin{equation*}
  \{r(y)\ge M\}\subset\mathcal{E}(y,1,M)\cup\mathcal{E}(y,2,M)
\end{equation*}
so that
\begin{equation}
S_{M}\le\sum_{i=1}^{2}4M^2\sum_{\ell\le d(o,y)\le\ell^{2}}\frac{\mathbbm{1}\{\mathcal{E}(y,i,M)\}}{d(o,y)^{2}}.\label{eq:SM by i}
\end{equation}
The proof of Lemma \ref{lem:subcriticality} also implies that, for $\eps\le\eps_0$,
\begin{equation}\label{eq:Eyi}
  \mathbb{P}(\mathcal{E}(y,i,M))\le e^{-M}.
\end{equation}
For $M$ small we also need the fact that for every $\delta>0$
there exists an $\varepsilon_{1}(\delta)$ such that $\varepsilon\le \varepsilon_{1}$($\delta)$ implies $\mathbb{P}(\mathcal{E}(y,i,M))\le\delta$, which holds for any $M\ge 1$. This is also proved exactly like Lemma \ref{lem:subcriticality}.

Going back to $S_{M}$, we further subdivide (\ref{eq:SM by i})
according to the value of the coordinates of $y$ modulo $3M$, defining
\[
S_{v,i,M}:=\sum_{\substack{\ell\le d(o,y)\le\ell^{2}\\
y\equiv v\mod 3M
}
}\frac{\mathbbm{1}\{\mathcal{E}(y,i,M)\}}{d(o,y)^{2}},\quad v\in[0,3M)^{2}.
\]
The events in this last sum are independent (each $\mathcal{E}(y,i,M)$
depends only on $\CP_i$ in $B(y,M)$ and these subsets are disjoint). Hence any of the standard methods may lead to the following estimate: for every $s>2\mathbb{E}S_{v,i,m}$,
\[
\mathbb{P}\Big(S_{v,i,M}>s\Big)\le C\exp(-cs\ell^{2}).
\]
(we used exponential moments, i.e.\ wrote $\mathbb{P}(S>s)\le\mathbb{E}(\exp(\mu(S-\mathbb{E}(S)))\linebreak[0]\exp(-\mu(s-\mathbb{E}(S)))$ with $\mu=c\ell^2$, but any other standard method would give a usable estimate).
Summing over $i$ and $v$ and using~\eqref{eq:SM by i} gives
\[
\mathbb{P}(S_{M}>18M^{2}s)\le CM^{2}\exp(-cs\ell^{2}).
\]
We use this inequality for $s=\frac{1}{36}M^{-3}\log\ell$, and note that
if $\varepsilon$ is sufficiently small then the condition $s>2\mathbb{E}S_{v,i,M}$
will be satisfied: indeed, for every $M\ge1$,
by summing (\ref{eq:Eyi}) over $y$, $\mathbb{E}S_{v,i,M}\le Ce^{-M}\log\ell$,
while for $M$ small the fact that $\mathbb{P}(\mathcal{E}(y,i,M)$)
can be made as small as needed by reducing $\varepsilon$ allows to make
$\mathbb{E}S_{v,i,M}\le C\delta\log\ell$
for any $\delta>0$. We get
\[
\mathbb{P}(S_{M}>\tfrac{1}{2M}\log\ell)\le CM^{2}\exp(-cM^{-3}\ell^{2}\log\ell).
\]
Summing over $M=1,2,4,\dotsc,2^{k}$ for $k=\lfloor\log_{2}\ell^{1/2}\rfloor$
gives
\[
\mathbb{P}\Big(\sum_{m=0}^{k}S_{2^{m}}>\log\ell\Big)\le C\exp(-c\ell^{1/2}\log\ell).
\]
Finally, the probability that $S_{M}>0$ for any $M\ge 1$ (in particular, for $M>2^{k}$) is no
more than $C\ell^{4}e^{-M}$ by~\eqref{eq:SM by at least M} and~\eqref{eq:Eyi}.
This establishes the lemma.
\end{proof}

\section{Proof of Lemma~\ref{lem:bad event prob estimates}}
Fix $\eps$ to be the minimum of the constants $\eps_0$ from Lemma~\ref{lem:subcriticality} and Lemma~\ref{lem:prob estimate on sum of radiuses}.
Fix $K=K(a_0)$ using proposition \ref{prop:perco} to
\begin{equation}\label{eq:K def}
K:= \max\big\{\sup\{K_{\textrm{Proposition \ref{prop:perco}}}(\eps,a)\colon a>a_0\},\,1\big\}.
\end{equation}
Recall from the addition algorithm the notations $\EC(u)$ \eqref{eq:EC_def} and $r(u,y)$ \eqref{def:rAndm} which we write here as $\EC_K$ and $r(y)$, respectively.
\def\rk{{r}}
We may apply to $\EC_K$ the probability estimates of Lemma~\ref{lem:subcriticality} and Lemma~\ref{lem:prob estimate on sum of radiuses} as, by Proposition~\ref{prop:perco}, $\EC_K$ is dominated by the union of two $\eps$-percolations and as $H$ is a subgraph of $\Z^2$.

The event $\gcal$ is comprised of 3 parts (recall Definition \ref{def:G}), and it will be convenient to name them $\gcal_2,\gcal_3,\gcal_4$ so, for example, $\gcal_2=\{u:\sqrt{J^+(u)J^-(u)}\ge d(o,x)^{-\sigma\lambda^2}\}$. We reserved $\gcal_1$ to the following auxiliary event (recall the definition of $M$ from \eqref{eq:M phi def}),
\[
\gcal_1=\{M(u|_H)\le\tfrac14\sqrt{d(o,x)}-2\}.
\]
Using Lemma~\ref{lem:subcriticality},
\begin{equation*}
  \P(\gcal_1^c)\le |H|\exp\left(-\tfrac14\sqrt{d(o,x)}-2\right)\le C\exp\left(-c\sqrt{d(o,x)}\right).
\end{equation*}
For the rest of the proof it is important to note that the bound~\eqref{eq:L lower bound} and the fact that $\lambda\le 1$ (by assumption) and $K\ge 1$ (by~\eqref{eq:K def}) show that
\begin{equation*}
  \gcal_1\subseteq\left\{M\left(u|_H\right)\le L(\tau, K)-2\right\}.
\end{equation*}
This is important because properties \ref{property:addition_lower_bound} and \ref{property:Jacobian_estimate} of the addition algorithm rely on this assumption, so most of the argument will work only on $\gcal_1$.

Let us start by bounding $\P(\gcal_3)$ (recall \eqref{eq:increment on boundary}). Let $y\in H$ satisfy $d(o,y)=\left\lfloor\frac{1}{2}d(o,x)\right\rfloor$. Property \ref{property:addition_lower_bound} of the addition algorithm implies that on $\gcal_1$, if $i_y\neq \tau_y$ then $\tau'(y,r(y))>0$. Since $\tau$ is constant on $B(y,\frac15 d(o,x))$ (see~\eqref{eq:tau choice}), this can only be if $r(y)\ge \frac{1}{5}d(o,x)$. Thus, relying again on Lemma~\ref{lem:subcriticality},
\begin{equation*}
  \P(\{i_y\neq \tau_y\}\cap\gcal_1)
  \le \P\left(r(y)\ge\tfrac15 d(o,x)\right)
  \le\exp\left(-\tfrac{1}{5}d(o,x)\right).
\end{equation*}

Before estimating $\gcal_2$ and $\gcal_4$ let us first discuss the discrete derivative of $\tau$. One checks that if $y,z$ are neighbours in $H$ then
$|\tau_y - \tau_z| \le 2\lambda/d(o,y)$.
Summing this gives (recall~\eqref{eq:tau_prime_def})
\begin{equation}\label{eq:tau long range differences}
  \tau'(y, k)\le \frac{3\lambda k}{d(o,y)},\quad 0\le k\le \sqrt{d(o,x)},
\end{equation}
using that $d(o,x)$ is sufficiently large. In addition,
\begin{equation}\label{eq:tau sum of squared differences}
  \sum_{\{y,z\}\in E(H)} (\tau_y-\tau_z)^2\le C\lambda^2\log(d(o,x)),
\end{equation}
again using that $d(o,x)$ is large. Lastly, the fact that $\tau_y=0$ when $d(o,y)<\sqrt{d(o,x)}$ implies that
\begin{equation}\label{eq:tau prime zero}
  \tau'(y,k)=0\quad \text{ when }\quad d(o,y)+k<\sqrt{d(o,x)}.
\end{equation}

We proceed to estimate $\P(\gcal_2)$ (recall \eqref{eq:Jacobian in E}). By~\eqref{eq:Jacobian estimate}, \eqref{eq:tau long range differences} and \eqref{eq:tau prime zero} we have on $\gcal_1$ that
\begin{equation*}
\begin{split}
  \sqrt{J^+(u)J^-(u)}
  &\;\stackrel{\mathclap{\textrm{\eqref{eq:Jacobian estimate}}}}{\ge}\;
  \exp\bigg(-\frac{1}{K^2}\sum_{y\in H} \tau'\Big(y,1+\max_{z\sim y}\rk(z)\Big)^2\bigg)\\
  &\;\stackrel{\mathclap{\textrm{(\ref{eq:tau long range differences},\ref{eq:tau prime zero})}}}{\ge}\;
   \exp\bigg(-\frac{9\lambda^2}{K^2}\sum_{d(o,y)\ge\frac{3}{4}\sqrt{d(o,x)}} \frac{(1+\max_{z\sim y}\rk(z))^2}{d(o,y)^2}\bigg).
\end{split}
\end{equation*}
For each $y\in H$,
\begin{equation*}
  \big(1+\max_{z\sim y}\rk(z)\big)^2
  \le 2+2\big(\max_{z\sim y}\rk(z)\big)^2
  \le 2 + 2\sum_{z\sim y} \rk(z)^2
\end{equation*}
whence
\begin{equation*}
\begin{split}
  \sum_{d(o,y)\ge\frac{3}{4}\sqrt{d(o,x)}} \frac{(1+\max_{z\sim y}\rk(z))^2}{d(o,y)^2}
  &\le C\sum_{d(o,y)\ge\frac{3}{4}\sqrt{d(o,x)}-1}
        \left(\frac{1}{d(o,y)^2} + \frac{\rk(y)^2}{d(o,y)^2}\right)\\
  &\le C\log(d(o,x)) + C\sum_{d(o,y)\ge\frac{3}{4}\sqrt{d(o,x)}-1}\frac{\rk(y)^2}{d(o,y)^2}.
\end{split}
\end{equation*}
so
\begin{align*}
  \lefteqn{\P(\gcal_2^c\cap\gcal_1)=\P\left(\Big\{\sqrt{J^+(u)J^-(u)}
    < \frac{1}{d(o,x)^{\sigma\lambda^2}}\Big\}\cap\gcal_1\right)}\quad&\\
  &\le \P\bigg(\exp\bigg(-\frac{9\lambda^2}{K^2}\bigg(C\log(d(o,x))+
  C\sum_{d(o,y)\ge\frac34\sqrt{d(o,x)}-1}\frac{r(y)^2}{d(o,y)^2}\bigg)\bigg)
  <\frac{1}{d(o,x)^{\sigma\lambda^2}}\bigg)\\
  &= \P\bigg(\sum_{d(o,y)\ge\frac{3}{4}\sqrt{d(o,x)}-1}\frac{\rk(y)^2}{d(o,y)^2} > \Big(\frac{\sigma K^2}{9C}-1\Big)\log(d(o,x))\bigg).
\end{align*}
Taking $\sigma$ large, as an absolute constant, will make $\sigma K^2/9C\ge \frac{3}{2}$ (recall that $K\ge 1$) and allow to apply Lemma~\ref{lem:prob estimate on sum of radiuses} (taking into account that $H$ is the induced subgraph of $\Z^2$ on $B(o,\frac{1}{2}d(o,x))$). We get that $\P(\gcal_1\setminus\gcal_2)\le C\exp(-c d(o,x)^{1/4})$.

We finally turn to the estimate of $\P(\gcal_4)$ (recall \eqref{eq:sum of squares of increments}). Property~\ref{property:addition_lower_bound} of the addition algorithm shows that on $\gcal_1$, for $y,z\in H$, $y\sim z$,
\begin{equation*}
\begin{split}
  (i_y - i_z)^2&\le \max\left\{\tau_y - \tau_z + \tau'(z, \rk(z)),\, \tau_z - \tau_y + \tau'(y, \rk(y))\right\}^2\\
  &\le 2\left((\tau_y-\tau_z)^2 + \tau'(y, \rk(y))^2 + \tau'(z, \rk(z))^2\right).
\end{split}
\end{equation*}
Thus, still on $\gcal_1$, we may use~\eqref{eq:tau sum of squared differences}, \eqref{eq:tau prime zero} and~\eqref{eq:tau long range differences} to obtain
\begin{equation*}
\begin{split}
  \sum_{\{y,z\}\in E(H)} (i_y - i_z)^2 &
  \;\stackrel{\textrm{\clap{\eqref{eq:tau sum of squared differences}}}}{\le}\;
  C\lambda^2\log(d(o,x)) + 8\sum_{y\in H}\tau'(y, \rk(y))^2\\
  &\;\stackrel{\textrm{\clap{(\ref{eq:tau long range differences},\ref{eq:tau prime zero})}}}{\le}\;
  C\lambda^2\log(d(o,x)) + 72\lambda^2\sum_{d(o,y)\ge\frac{3}{4}\sqrt{d(o,x)}}\frac{\rk(y)^2}{d(o,y)^2}.
\end{split}
\end{equation*}
Again, taking $\sigma$ large, as an absolute constant, we deduce from Lemma~\ref{lem:prob estimate on sum of radiuses} that
\begin{multline*}
  \P(\gcal_4^c\cap\gcal_1)
  =\P\Bigg(\Big\{\sum_{\substack{y,z\in H\\y\sim z}} (i_y - i_z)^2
    > \sigma\lambda^2\log d(o,x)\Big\}\cap\gcal_1\Bigg)\\
    \le \P\Bigg(\sum_{d(o,y)\ge\frac{3}{4}\sqrt{d(o,x)}}\frac{\rk(y)^2}{d(o,y)^2}
    > \frac{\sigma-C}{C}\log(d(o,x))\Bigg)\le C\exp(-cd(o,x)^{1/4}).
\end{multline*}
Combining the estimates on the probabilities of $\gcal_2$, $\gcal_3$ and $\gcal_4$, the lemma is proved. \qed

\section*{Acknowledgements}
We are grateful to Thomas Spencer who introduced us to the question of the decay rate of the weights in the VRJP model and suggested that techniques of Mermin--Wagner type may be applicable due to the log-convexity of the determinant. Michael Aizenman explained to RP the idea that fluctuation lower bounds and a priori upper bounds on the field can lead to a proof of decay, as implemented in~\cite{S09}. We thank Christophe Sabot for encouragement in the writing process.
GK is supported by the Israel Science Foundation, by the Jesselson Foundation and by Paul and Tina Gardner. RP is supported by the Israel Science Foundation grants 861/15 and 1971/19 and by the European Research Council starting grant 678520 (LocalOrder).

\end{document}